\documentclass[12pt,leqno]{amsart}
\usepackage{amscd}
\usepackage{amssymb}
\usepackage{amsfonts}
\usepackage{latexsym}
\usepackage{amsthm}

\newcommand{\1}{$_\mathrm{a}$}
\newcommand{\2}{$_\mathrm{b}$}
\newcommand{\3}{$_\mathrm{c}$}
\newcommand{\tp}{^{\!\top}\!}
\newcommand{\acs}{almost-complex structure}
\newcommand{\aK}{almost-K\"ahler\ }
\newcommand{\pd}{\partial}
\newcommand{\Nil}{\mathit{Nil}}
\newcommand{\Sol}{\mathit{Sol}}
\newcommand{\ext}[1]{\raise1pt\hbox{$\bigwedge^{\!#1}$}\,}
\renewcommand{\ge}{\geqslant}
\renewcommand{\le}{\leqslant}

\numberwithin{equation}{section}
\theoremstyle{plain}
\newtheorem{theorem}{Theorem}[section]
\newtheorem{lemma}[theorem]{Lemma}
\newtheorem{prop}[theorem]{Proposition}
\newtheorem{cor}[theorem]{Corollary}
\theoremstyle{remark}
\newtheorem{rem}[theorem]{Remark}

\newtheorem{ex}[theorem]{Example}
\newcommand\C{\mathbb{C}}
\newcommand\R{\mathbb{R}}
\newcommand\Z{\mathbb{Z}}
\newcommand{\cR}{\mathcal{R}}
\newcommand{\bJ}{\mathbb{J}}
\newcommand{\bT}{\mathbb{T}\kern1pt}
\newcommand{\eE}{\mathrm{e}}
\newcommand{\eF}{\mathrm{e}^F}
\newcommand{\bk}{\overline{k}}
\newcommand{\bl}{\overline{l}}
\newcommand{\la}{\lambda}
\newcommand{\Ga}{\Gamma}
\newcommand{\La}{\Lambda}
\newcommand{\Om}{\Omega}
\newcommand{\hOm}{\widehat\Omega}
\newcommand{\tOm}{\tilde\Omega}
\newcommand{\ta}{\tilde a}
\newcommand{\hg}{\widehat g}
\newcommand{\Th}{\Theta}
\newcommand{\tTh}{\widetilde\Th}
\newcommand{\we}{\wedge}

\begin{document}
\parskip2pt

\title{The Calabi--Yau equation on 4-manifolds over 2-tori}

\author{A. Fino, \ Y.Y.~Li, \ S.~Salamon, \ L.~Vezzoni}

\begin{abstract}
  This paper pursues the study of the Calabi--Yau equation on certain
  symplectic non-K\"ahler 4-manifolds, building on a key example of
  Tosatti--Weinkove \cite{TW} in which more general theory had proved
  less effective. Symplectic 4-manifolds admitting a 2-torus fibration
  over a 2-torus base $\bT^2$ are modelled on one of three solvable
  Lie groups. Having assigned an invariant \aK structure and a volume
  form that effectively varies only on $\bT^2$, one seeks a symplectic
  form with this volume. Our approach simplifies the previous analysis of
  the problem, and establishes the existence of solutions in various
  other cases.
\end{abstract}

\maketitle

\section{Introduction}

Let $(M,J,\Om)$ be a $2n$-dimensional compact K\"ahler manifold with
associated complex structure $J$ and symplectic form $\Om$. Yau's
theorem implies that any representative of the first Chern class of
$M$ can be written as the Ricci curvature of a unique K\"ahler metric
whose 2-form $\tOm$ lies in the same cohomology class of $\Om$
\cite{Yau}. This can be restated by saying that for every volume form
$\sigma$ on $M$ satisfying
\begin{equation}\label{CY1}
\int_M \sigma=\int_M \Om^n,
\end{equation}
there exists a unique K\"ahler form $\tOm$ such that $[\tOm]=[\Om]$ and
\begin{equation}\label{CY2}
\tOm^n=\sigma\,.
\end{equation}
\smallbreak

The same problem can be posed in an `almost-K\"ahler' context, when
$\Om$ remains closed but $J$ is merely an \acs. The latter is still
orthogonal relative to a Riemannian metric $g$ for which
$\Om(X,Y)=g(JX,Y)$, and
\begin{equation}\label{CY3}
\tOm=\Om+d\alpha
\end{equation}
is again assumed to be a positive-definite $(1,1)$-form relative to
$J$. In this set-up, \eqref{CY1},\,\eqref{CY2},\,\eqref{CY3}
constitute the \emph{Calabi--Yau} problem. Donaldson proved in
\cite{Donaldson} that its solution is unique in four dimensions
(see Proposition~3.2), and we provide a dimensionally reduced version of
this proof in Section~\ref{Kodaira}. In the \aK and Hermitian cases,
the Calabi-Yau problem has been further studied in extensive papers by
Tosatti, Weinkove, Yau and others \cite{TWY,TW2}.

A sufficient condition for the existence of solutions to the
Calabi--Yau equation in terms of the Chern connection was given in
\cite{TWY}, but this condition fails in the strictly \aK case when a
scalar curvature function is negative (Proposition~\ref{s<0}). The
simplest such case is that of a Kodaira surface. The latter is a
discrete quotient of $\C^2$ and carries a holomorphic symplectic form
trivializing the canonical bundle \cite{Kod}. It thereby becomes a
symplectic non-K\"ahler manifold that is realized as a principal
$T^2$-bundle over a 2-torus base $\bT^2$. (Throughout this paper, we
use blackboard font for the base, to avoid confusion.)  As such, it is
often called the Kodaira--Thurston manifold \cite{Th}, and is a
valuable vehicle for testing new phenomena in differential geometry
\cite{Abb,FPPS}.

In \cite{TW}, Tosatti and Weinkove solved the Calabi--Yau problem on
the Kodaira--Thurston manifold $M$ for every $T^2$-invariant volume
form $\sigma$. The purpose of this paper is threefold:

\noindent(i) To extend the analysis of \cite{TW} to emphasize
invariance of the solution by the $S^1$ action that rotates the
symplectic forms $\Om_2,\Om_3$ and associated \acs s $J_2,J_3$.

\noindent(ii) To generalize the theory so as to tackle other
$T^2$-bundles over $\bT^2$ and assigned volume forms from the base.

\noindent(iii) To set up the problem with arbitrary invariant metrics,
as a first step to tackling more general initial data.

We accomplish (i) with the aid of Theorem~\ref{KT}, reducing the
problem to a standard Monge--Amp\`ere equation on the base, which can
then be solved by an established theorem of the second author
\cite{Li}. This is based on the existence of a potential for the
holomorphic syplectic structure, which is not surprising given that
the relevant fibration is analogous to a moment mapping. The approach
makes the solutions (or at least, the \emph{method} of solution) more
explicit, and this is our aim in the other cases.

By work of Ue \cite{Ue}, any orientable 2-torus bundle $M$ over a
2-torus is a smooth quotient $\Gamma\backslash X$ covered by
solvmanifold. In practice, there are just three candidates for $X$
that we need consider in (ii), namely $\Nil^3 \times \R$, $\Nil^4$ and
$\Sol^3\times \R$, all diffeomorphic to $\R^4$, in which
$\Nil^3,\Nil^4$ are nilpotent Lie groups and $\Sol^3$ is a particular
solvable Lie group. It is well known (and we shall demonstrate) that
such quotients all admit symplectic structures. The notion of
\emph{invariant} \aK structure makes sense in this context, meaning
one induced from a left-invariant structure on $X$ which is invariant
by $\Gamma$. As regards (iii), we make most progress in the nilpotent
case:

\begin{theorem}\label{main}
Let $M=\Gamma \backslash X$ with $X=\Nil^3\times\R$ or $\Nil^4$, and
suppose that $M$ admits an invariant \aK structure $(g,J,\Om)$, and a
$\,T^2$-fibration $\pi\colon M\to \mathbb T^2$ whose fibres are
Lagrangian. Then for every normalized volume form $\sigma=\eE^F\Om^2$
with $F\in C^\infty(\bT^2)$, the corresponding Calabi--Yau problem has
a unique solution.
\end{theorem}

The Lagrangian condition may or may not apply in the 2-step case, but
is automatic when $M$ is modelled on $\Nil^4$. The latter case leads
to a generalized Monge--Amp\`ere equation on the 2-torus base, for
which we establish solutions in Section~\ref{GMA}. The same equation
is needed to establish Theorem~\ref{main} for a second fibration
$\Ga\backslash(\Nil^3\times\R)\to\bT^2$ in which the Lagrangian
condition is not automatic, and in Section~\ref{2fib}, we exhibit this
2-step situation as a limiting case of the 3-step one.

Compact quotients of $Sol^3\times \R$ do not admit Lagrangian
fibrations, and therefore Theorem \ref{main} does not apply. On the
other hand, if $\Om$ is a symplectic form on the total space of any
$T^2$-fibration $M^4\to\bT^2$ for which the fibres are not Lagrangian,
then we can choose a compatible \acs\ $J$ such that $\pi$ is
$J$-holomorphic. An elementary argument (Proposition~\ref{simp}) shows
that the Calabi--Yau problem for $(g,J,\Om)$ also has a solution in
this very special case. We extend this argument in Section~\ref{Sol}
by considering a foliation of $M=\Gamma\backslash(\Sol^3\times\R)$ by
2-dimensional leaves transverse to $\pi$ that are holomorphic relative
to \emph{any} invariant \aK structure.

\smallbreak\noindent\textit{Acknowledgments.} The authors thank Pawel
Nurowski for a crucial observation that led to the reduction
accomplished in Section~\ref{Kodaira}, and Valentino Tosatti for his
encouragement. Aspects of the work were described at the
\textit{Workshop on K\"ahler and related geometries} (Nantes 2009) and
\textit{XIX International fall workshop on geometry and physics}
(Porto 2010). The paper was completed while the third author was a
visitor at the IH\'ES.

\section{A brief classification of $T^2$-bundles over $\bT^2$}
\label{classification}

Orientable bundles with a 2-torus fibre $T^2$ over a 2-torus base
$\bT^2$ were classified by Sakamoto and Fukuhara in
\cite{Sakamoto-Fukuhara} and it was shown by Ue that all these
manifolds are geometric. By a \emph{geometric $4$-manifold} in the
sense of Thurston one means a pair $(X, G)$ where $X$ is a complete,
simply-connected Riemannian $4$-manifold, $G$ is a group of isometries
acting transitively on $X$ that contains a discrete subgroup $\Ga$
such that $\Ga\backslash X$ has finite volume. Since the stabilizer of
$G$ at a point in $X$ is compact and so $\Ga\backslash X$ is compact
if and only if $\Ga\backslash G$ is compact.

Two pairs $(X, G)$ and $(X', G')$ define the same geometry if the
actions of $G$ and $G'$ are related by a diffeomorphism $f: X
\rightarrow X'$. There are nineteen $4$-dimensional geometries $X$
which have been classified by Filipkiewicz \cite{Fi}.  In the paper we
will only consider the ones which are orientable $T^2$-bundles over
$\bT^2$ and we will take $G$ to be the identity component of the
isometry group of $X$.  The relevant pairs $(X,G)$ are:

\begin{itemize}

\vspace{2mm}\item[(0)] $X = E^4$ (Euclidean $4$-space) and $G =
SO(4)\ltimes\R^4$, the semidirect product of translations and
rotations;

\vspace{2mm}\item[(1)] $X =\Nil^3\times E^1$, where $\Nil^3$ is the
$3$-dimensional Heisenberg group of matrices
\begin{equation}\label{xyz} H_{x,y,z}=
\left(\!\begin{array}{ccc}1&x&z\\0&1&y\\0&0&1\end{array}\!\right)
\end{equation}
under multiplication. Here, $G$ is spanned by $\Nil^3$ acting on
itself on the left, by translation in the $E^1$-factor and by an
additional isometric action of $S^1$:
\[\textstyle (x, y, z) \mapsto (x \cos \theta + y \sin \theta, - x
  \sin \theta + y \cos \theta, z + \frac12 \sin\theta (y^2 \cos \theta
  - x^2\cos \theta - 2 xy \sin \theta));\]

\vspace{2mm}\item[(2)] $X=\Nil^4 = \R\ltimes\R^3$ is the group of
real matrices
\begin{equation}\label{consist}
\left(\begin{array}{cccc} 1&t&\frac 12
  t^2&x\\ 0&1&t&y\\ 0&0&1&z\\ 0&0&0&1 \end{array} \right )
\end{equation}
under multiplication. Given that $\Nil^4$
coincides with the identity component of its isometry group, $G=X$
acts on itself on the left.

\vspace{3mm}\item[(3)] $X = \Sol^3\times E^1$, where $\Sol^3 =
\R\ltimes_\varphi\R^2$ is a solvable Lie group with $\varphi(t) =
\hbox{\small$\left(\!\begin{array}{cc} \eE^t &0\\0&\eE^{-t}\end{array}
  \!\!\right)$}$, and $G = X$ acts by left multiplication:
\[(x_0, y_0, z_0, t_0) (x, y, z, t) = (x_0 + \eE^{t_0} x ,\ y_0 +
  \eE^{-t_0} y,\ z_0 + z,\ t_0 + t).\]
\end{itemize}

\bigbreak

By a suitable choice of basis, we can specify the Lie algebras of the
three non-abelian goups as follows:
\begin{equation}\label{str}
\begin{array}{lccc}
{\rm(1)} & \Nil^3\times \R &\longleftrightarrow& (0,0,0,12),\\[8pt]
{\rm(2)}\qquad & \Nil^4 &\longleftrightarrow& (0,13,0,12),\\[8pt]
{\rm(3)} & \Sol^3 \times \R &\longleftrightarrow& (0,0,13,41).
\end{array}
\end{equation}
The digits on the right encode the exterior derivative relative to an
invariant coframe $(e^1,\dots,e^4)$, and in any case we shall use the
notation $e^{12\cdots} = e^1\we e^2\we\cdots$ for simple
differential forms. For example, in case (2), we opt to define
$e^1=-dt$, $e^2=dy-tdz$, $e^3=dz$, and (see \eqref{-tyz}) we obtain
\begin{equation}\label{3st}
de^i=\left\{\begin{array}{ll}
0 & i=1,3,\\[3pt]
e^{13},\quad & i=2,\\[3pt]
e^{12}, & i=4.
\end{array}\right.
\end{equation}
If one re-scales so that the middle line becomes $de^2=\la e^{13}$,
case (1) is the result of letting $\la$ tend to 0. This observation is
important for the sequel.

Note that $\Sol^3$ refers to just one particular solvable Lie group,
distinct for example from that with Lie group corresponding to
$(0,13,14)$ that has no cocompact lattice.\smallbreak 

The diffeomorphism classes of (the total space of) $T^2$-bundles over
$\bT^2$ can be summarized in Geiges' eight families
\cite[Table~1]{Geiges}. By \cite[Theorem 3]{Geiges}, the total space
is K\"ahler if and only the geometric type is $E^4$; this happens for
two families, which are not interesting from our point of view. For
each of the other six families we next describe $X$, $\Gamma$, the
monodromy matrices $A_1,A_2\in SL(2, {\mathbb Z})$ along the curves
$\gamma_1,\gamma_2$ generating $\pi_{1}(\bT^2)$, as well as the Euler
class $(m,n)$ for the corresponding $T^2$-bundle, which can be then
described by the generators of $\Ga$. By \cite{Ue2} for some lifts
$\tilde \gamma_i$ to $M$ we have
\[\pi_{1}(M)=\left<\tilde \gamma_1, \tilde \gamma_2, l, h\ \vert\ [l,
  h]=1,\ \tilde\gamma_i(l, h)\gamma_i^{-1}=(l,h)A_i,\ [\tilde
  \gamma_1,\tilde \gamma_2]=l^{m}h^{n}\right>.\]
  The discrete group
  $\Gamma$ arises as the image of a faithful representation $\rho$
  from $\pi_{1}(M)$ to the group of orientation-preserving isometries
  of $X$, so $\Gamma$ is generated by $\rho(\tilde \gamma_1)$,
  $\rho(\tilde\gamma_2)$, $\rho (l), \rho(h)$.  To simplify the notation we
  identify a generating map with the image of $(x,y,z,t)$ under this
  map. Unless otherwsie stated, there is no restriction on $(m,n)$. 
  Following \cite{Geiges} and \cite{Ue2}, we have the six families:

\begin{enumerate}

\vspace{2mm}
\item[(1\1)] $X = \Nil^3\times E^1$, with $\Ga$ generated by
\[\begin{array}{l}(x+1,y,z+\lambda y,t),\ (x,y+1,z,t),\ (x,y,z+1,t),\
  (x,y,z,t+1).\end{array}\]
Here $A_1=A_2= I$ (the identity matrix), $\la\ne0$ and $(m, n) \neq (0,0)$.

\vspace{5mm}
\item[(1\2)] $X = \Nil^3\times E^1$, with $\Ga$ generated by
\[\begin{array}{l}
(x,y,z+\beta_1,t),\ (x+\alpha_1, y,z+\alpha_1y,t),\\[3 pt]
(x+\alpha_2,y,z+\alpha_2 y+\beta_2,t+1),\
  (-x,y+1,-z+\gamma_3,t).\end{array}\]
Here $A_1= \left ( \begin{array}{cc} -1& \lambda\\ 0&-1 \end{array} \right )$, 
$A_2 = I,$
with $\beta_1,\alpha_1,\lambda$ all non-zero.

\vspace{5mm}
\item[(1\3)] $X = \Nil^3\times E^1$, with $\Ga$ generated by
\[\begin{array}{l}
(x,y,z+\beta_1,t),\,(x+\alpha_1,y,z+\alpha_1y,t),\\[5 pt]
(x+\alpha_3,y+1,z+\alpha_3 y+\beta_4,t),\,(-x,y,-z,t+1).\end{array}\]
Here $A_1 = \left ( \begin{array}{cc} 1& \lambda\\ 0& 1 \end{array}
\right ),$ $A_2=  -I$,
with $\beta_1,\alpha_1,\lambda$ all non-zero.

\vspace{5mm}
\item[(2)] $X = \Nil^4$, with $\Ga$ generated by
  \[ \begin{array}{l} (x+\alpha_1,y,z,t),\quad
    (x,y+\beta_1,z,t),\\[3pt] (x+y+\frac12z+\alpha_2,y,z,t+1),
    \quad(x+\alpha_3y+\beta_3,z+1,t).\end{array}\] 
Here $A_1 = \left ( \begin{array}{cc} 1& \lambda\\ 0& 1 \end{array}
    \right ),$ $A_2 = I$, with $\beta_1 = \frac1n$, 
$\lambda = \frac{1}{n \alpha_1}$, $\beta_3 =
  \frac{m}{\alpha_1} - \frac12$ and $n\ne0$.  

\vspace{2mm}
\item[(3\1)] $X = \Sol^3\times E^1$, with $\Ga$ generated by
\[\begin{array}{l} (x+\alpha_1,y+\beta_1,z,t),\qquad
  (x+\alpha_2,y+\beta_2,z,t),\\[3pt]
 (x+\alpha_3,y+\beta_3,z + 1,t),\quad (\epsilon\,
\eE^\delta x+\alpha_4,\epsilon\,\eE^{-\delta}y+\beta_4,
z, t + \delta).\end{array}\]
Here $A_1\in SL(2, {\mathbb Z})$, $A_2 =  I,$ and $\epsilon=\pm 1$.

\vspace{5mm}
\item[(3\2)] $X = \Sol^3\times E^1$, with $\Ga$ generated by
\[\begin{array}{l} (x+\alpha_1,y+\beta_1,z,t),\qquad
  (x+\alpha_2,y+\beta_2,z,t),\\[3pt]
 (-x+\alpha_3,-y+\beta_3,z + 1,t),\quad (\epsilon\,
\eE^\delta x+\alpha_4, \epsilon\,\eE^{-\delta}y+\beta_4,
z, t + \delta),.\end{array}\]
Here $A_1\in SL(2, {\mathbb Z})$, $A_2= -I,$ and $\epsilon= 1$. As in (3\1),
there are various relations (which we omit) between the other parameters.

\end{enumerate}

\bigbreak

The total space $M$ of an orientable $T^2$-bundle over $\bT^2$ is then
a quotient $\Ga\backslash X$, where $X\cong\R^4$ is taken from the
list above and $\Ga$ is a discrete subgroup which acts freely on $X$.
Although $\Gamma$ is not necessarily a subgroup of $X$, it always
contains a subgroup $\Gamma_0$, which is a lattice of $X$, such that
$\Gamma_0\backslash \Gamma$ is finite. As such, $\Gamma \backslash X$
is covered by the solvmanifold $\Gamma_0\backslash X$. When
$\Gamma=\Gamma_0$ the quotient $\Gamma \backslash X$ is a
\emph{solvmanifold} (resp.\ \emph{nilmanifold} if $X$ is nilpotent),
otherwise $\Gamma\backslash X$ is called an \emph{infra-solvmanifold}
(resp.\ \emph{infra-nilmanifold}). In the list above, in the cases
(1\1) and (2), the group $\Gamma$ is always a lattice, while in the
case (3\1) $\Gamma$ is a lattice if and only if its last generator has
$\epsilon=1$.  In the cases (1\2) and (1\3) the quotient
$\Gamma\backslash X$ is an infra-nilmanifold such that
$\Gamma_0\backslash \Gamma=\Z_2$, and in (3\2) $\Gamma$ is not
a lattice of $X$.

\begin{ex}\label{Ex1}
In case (3\1), take $\alpha_4=\beta_4=0$
and $\epsilon=1$.  We also take $\alpha_3=\beta_3=0$ so that
$M=M^3\times S^1$ is the product of a solv-3-manifold with a
circle. The lattice in $\R^2$ defining $\bT^2$ will be generated by
\[
\left(\!\begin{array}{c}\alpha_1\\\beta_1\end{array}\!\right) =
\left(\!\begin{array}{c}1\\1\end{array}\!\right),\qquad
\left(\!\begin{array}{c}\alpha_2\\\beta_2\end{array}\!\right) =
\left(\!\begin{array}{c}\sigma\\\sigma^{-1}\end{array}\!\right),\]
where $\sigma$ and $\sigma^{-1}$ are the two roots of the quadratic
equation
$\sigma^2-n\sigma + 1=0$, with $n=3,4,5$.
Setting $\delta=\log\sigma$ ensures that
\[
\left(\!\begin{array}{cc}\eE^\delta&0\\0&\eE^{-\delta}\end{array}\!\right)
= \left(\!\begin{array}{cc}1&\sigma\\1&\sigma^{-1}\end{array}\!\right)
\left(\!\begin{array}{cc}0&\!-1\\1&n\end{array}\!\right)
\left(\!\!\begin{array}{cc}1&\sigma\\1&\sigma^{-1}\end{array}\!\!\right)^{\!-1},\]
so $\Gamma$ is closed under multiplication. The example on \cite[page
  25]{Bock} is the instance $n=3$.
\end{ex} 

The inequivalent fibrations $\pi\colon M\to\bT^2$ are induced from the
following coordinate mappings:
\begin{equation}\label{mappings}
(x,y,z,t)\ \mapsto\ \left\{
\begin{array}{cl}
(x,y)\hbox{ or }(y,t) & \hbox{for (1\1)},\\[5pt]
(y,t) & \hbox{for (1\2), (1\3)},\\[5pt]
(z,t) & \hbox{for (2), (3\1) and (3\2)}.
\end{array}\right.
\end{equation}
In view of \cite[Theorem 1]{Geiges}, the total space $M$ always admits
a symplectic structure, and we shall check this case by case early in
the next section.

\section{\aK structures}
\label{AK}

Since our main focus of attention is real symplectic geometry, we
first discuss the existence of invariant closed non-degenerate 2-forms
compatible with each of the fibrations listed in \eqref{str}. Any such
2-form can always be expressed as $f^{13}+f^{42}$ relative to some
basis $(f^i)$ related by a constant linear transformation to a basis
$(e^i)$ defined by \eqref{str}. It follows that we can choose a
Riemannian metric $g$ (for example $\sum f^i\otimes f^i$) and an
orthogonal \acs\ $J$ (for example,
\[ J= f^1\otimes f_3-f^3\otimes f_1-f^2\otimes f_4+f^4\otimes f_2),\]
such that
\[ \Om(X,Y)=g(JX,Y).\]
The resulting triple $(g,J,\Om)$ constitutes an invariant
\aK structure on $M$.

Returning to \eqref{str}, we first remark that $M$ has a complex
structure if and only if $b_1(M)=3$, which is precisely case (1\1) of
the previous section \cite{Geiges}. We shall denote the two
inequivalent projections from the first row of \eqref{mappings} by
\begin{equation}\label{pi's}
\begin{array}{rcl}
\pi_{xy} &\colon& M\to \bT^2_{xy},\\[5pt]
\pi_{yt} &\colon& M\to \bT^2_{yt}.
\end{array}
\end{equation}
The usual projection $\pi_{xy}$ defining the Kodaira--Thurston
manifold $M$ is holomorphic with respect to an invariant complex
structure $J_1$. There is also a holomorphic symplectic 2-form
$\Th=\Om_2+i\Om_3$ for which $\pi_{xy}$ is \emph{bilagrangian},
meaning that both $\Om_2$ and $\Om_3$ restrict to zero on the fibres.
This is equivalent to the single complex equation $\Th\we
e^{12}=0$. With the labelling below, the canonical choice is
$\Om_2=e^{13}+e^{42}$ and $\Om_3=e^{14}+e^{23}$ (see \eqref{basisKT}
and \eqref{Theta0}), though one aspect of Section~\ref{Kodaira} is to
explain that there are families of equivalent such structures.

With the above convention, $\Om_2\we e^{13}\ne0$, so $\Om_2$ is
non-degenerate on the fibres of $\pi_{yt}$. Indeed, $\Om_2$ is
compatible with the \acs\ characterized by
\[ Je^1=e^3,\qquad Je^4=e^2,\]
for which $\pi_{yt}$ is holomorphic. By this we mean that $J$ induces
a complex structure $\bJ$ on $\bT=\bT^2_{yt}$ and the differential
\begin{equation}\label{JJ}
\pi_*\colon (T_mM,J)\to(\bT,\bJ)
\end{equation}
is complex linear ($\bJ$ is necessarily integrable for dimensional
reasons). This situation, in which the projection respects the
\emph{\aK}structure, contrasts with that for $\pi_{xy}$. On the other
hand, the fibres of $\pi_{yt}$ remain Lagrangian for $\Om_3$. We shall
study $(M,\Om_3)$ as a limiting case of the $\Nil^4$ model, in which
$e^{13}$ is again the volume form on the base and $\Om_3$ is a
symplectic form for which the fibres are necessarily Lagrangian.

In the cases different from (1\1), the fibration of $M$ as a torus
bundle is unique \cite[Theorem~1]{Geiges}. In (3\1) and
(3\2), the symplectic form $e^{12}+e^{34}$ is non-zero (and so
non-degenerate) on the fibres of $\pi_{zt}$. Note that $e^{34}$ is
closed by \eqref{str}(3), in which the distinction between $e^{41}$
and $e^{14}$ is important. The following general result is relevant to
this situation.

\begin{prop}\label{simp}
Let $\Om$ be an invariant symplectic form on $M=\Gamma\backslash X$,
and suppose that $\Om$ is non-zero on the fibres of a $T^2$-fibration
$\pi\colon M\to\bT^2$. Then there exists an invariant \aK structure
$(g,J,\Om)$ on $M$ for which $\pi$ is $J$-holomorphic. In this
situation, if $\sigma=\eF\,\Om^2$ is a volume form satisfying
\eqref{CY1} with $F\in C^\infty(\bT^2)$, the associated Calabi--Yau
problem has a solution.
\end{prop}

\begin{proof} Let $f^{12}$ be a volume form on the base $\bT^2$,
and fix a point $m\in M$. By assumption, $f^{12}\we\Om\ne0$ and we can
choose $\mu\in\R$ such that
\[ 0=(\Om-\mu f^{12})^2 = \Om^2-2\mu f^{12}\we\Om\] at $m$.
This means that $\Om-\mu f^{12}$ is a simple 2-form, and so there
exists a basis $(f^i)$ of 1-forms such that
\begin{equation}\label{Ommu}
\Om = \mu f^{12}+f^{34}.
\end{equation}
The required \aK structure is obtained by defining $J$ to be the \acs\
$J$ whose value at $m$ is given by $Jf_1=\mu f_2$ and $Jf_3=f_4$ in
the dual basis. As above, $J$ induces a complex structure $\bJ$ on
$\bT^2$, and \eqref{JJ} is complex linear.

As $m$ varies along the fibre, the parameter $\mu$ remains constant by
invariance. To solve the Calabi--Yau problem, we merely set
\[\tOm = \eF\mu f^{12}+f^{34},\] which yields \eqref{CY2} immediately.
Observe that $\tOm-\Om$ is the pullback of the $(1,1)$-form
$(\eF-1)\mu f^{12}$ on $\bT=\bT^2$. It remains to show that it
evaluates to zero on the fundamental cycle $[\bT]\in
H_2(\bT,\Z)$, but
\[\int_\bT(\eF-1)\mu f^{12} = \int_M(\eF-1)\mu f^{1234} =
    {\textstyle\frac12}\!\int_M(\sigma-\Om^2)=0\,,\] by hypothesis.
\end{proof}

For completeness, we reproduce the `negative chords' argument from
\cite{Donaldson} that guarantees that the elementary solution above is
unique. We shall tag this uniqueness statement to all our own
existence results without further comment (except for
Remark~\ref{!!}).

\begin{prop}\label{!} Any solution of the Calabi--Yau problem
\eqref{CY1}, \eqref{CY2}, \eqref{CY3} on a compact 4-manifold is
unique.
\end{prop}

\begin{proof}
  Given two solutions $\tOm,\tOm'$ of the problem, set
  $\tOm-\tOm'=d\alpha$ and $\tOm+\tOm'=\hOm$. With respect to $J$,
  both $d\alpha$ and $\hOm$ are of type $(1,1)$ and the latter defines
  a (positive-definite) Riemannian metric $\hg$ by the usual formula
  $\hOm(X,Y)=\hg(JX,Y)$. Relative to $\hg$, we have the decomposition
\begin{equation}\label{ext2}
\ext2 T^*_mM = \La^+\oplus\La^-
\end{equation}
into self-dual and anti-self-dual forms at each point, and the
$(1,1)$-forms are generated by $\La^-$ together with $\hOm$.  But
since $\hOm\we d\alpha=0$, we can conclude that $d\alpha\in\La^+$,
and $d\alpha\we d\alpha$ is $\|d\alpha\|^2$ times
$\mathrm{vol}_{\hg}$. A standard Stokes' argument then tells us
that $d\alpha=0$.
\end{proof}

\subsection{The condition $\cR>0$.} Tosatti, Weinkove and Yau
defined a certain tensor $\cR$ using the canonical connection and
Nijenhuis tensor of $J$ \cite{TWY}.  They proved that if $\cR$ is
positive, then the Calabi--Yau equation has a solution.

On any \aK manifold $(M,J,\Om)$ there exists a unique connection
$\nabla$ such that $\nabla J=0$, $\nabla\Om=0$ and whose has no
$(1,1)$-component. This is the so-called \emph{canonical connection}
or \emph{Chern connection} (though the latter is sometimes reserved
for the Hermitian case). We denote its curvature tensor by $R$, and
set $$\cR_{i \overline{j} k \overline{l}}(g, J) = R_{ik
  \overline{l}}^j + 4 N_{\overline{l} \overline{j}}^r
\overline{N_{\overline{r} \overline{k}}^i}\,, $$
where $N$ is the
Nijenhuis tensor of $J$. According to \cite{TWY}, the condition
\begin{equation} \label{yau}
{\mathcal {R}} (g, J) \ge 0,
\end{equation}
 implies that the Calabi--Yau equation can be solved for every
 normalized volume form.

It was pointed out in \cite{TW} that \eqref{yau} fails to hold for the
Kodaira--Thurston manifold. In fact, it is never satisfied under the
hypothesis of Theorem \ref{main}. To see this, we first prove

\begin{prop}\label{s<0}
Let $(M,J,\Om)$ be a $4$-dimensional \aK manifold and
let $s=g^{r\bk}g^{s\bl}R_{r\bk s\bl}$ be the scalar curvature of $\nabla$.
Assume $s\le0$. Then condition \eqref{yau} is satisfied if and only
if $(J,\Om)$ is a K\"ahler structure.
\end{prop}

\begin{proof}
Let ${Z_1,Z_2}$ be a local unitary frame on $M$. Then direct
computation gives
\[\begin{array}{rcl}
\cR_{1\overline11\overline1} &=& R_{11\overline1}^1,\\[3pt]
\cR_{2\overline22\overline2} &=& R_{22\overline2}^2,\\[3pt]
\cR_{1 \overline 1 2 \overline 2} &=&
R_{1 2 \overline 2}^{1}+4N_{\bar 2 \bar 1}^1 N_{1 2}^{\bar 1}\ =\
R_{2 1 \overline 1}^{2}-4 |N_{\bar 1 \bar 2}^1|^{2}\,,\\[3pt]
\cR_{2 \overline 2 1 \overline 1} &=&
R_{2 1 \overline1}^{2}+4N_{\bar 1 \bar 2}^2 N_{2 1}^{\bar 2}\ =\
R_{2 1 \overline1}^{2}-4 |N_{\bar 1 \bar 2}^2 |^{2}\,.
\end{array}\]
These equations imply
\[\begin{array}{l}

\cR_{2 \overline 2 2 \overline 2}+\cR_{1 \overline 1 2
  \overline 2}\ =\ {\rm Ric}(J)_{2\bar 2}-4|N_{\bar 1 \bar 2}^1
|^{2}\,,\\[3pt] \cR_{1 \overline 1 1 \overline
  1}+\cR_{2 \overline 2 1 \overline 1}\ =\ {\rm Ric}(J)_{1\bar
  1}-4|N_{\bar 1 \bar 2}^2 |^{2},
\end{array}\]
and
$$ \cR_{2 \overline 2 2 \overline 2}+\cR_{1 \overline
  1 2 \overline 2}+\cR_{1 \overline 1 1 \overline
  1}+\cR_{2 \overline 2 1 \overline 1}=s-4|N_{\bar 1 \bar 2}^1
|^{2}-4|N_{\bar 1 \bar 2}^2 |^{2}\,.
$$ Hence if $s$ is non-negative and condition \eqref{yau} is
satisfied, then $N$ vanishes.
\end{proof}

\begin{cor}
On a $4$-dimensional infra-solvmanifold $M=\Gamma\backslash G$ with
an invariant \aK structure condition \eqref{yau} is
satisfied only in the K\"ahler case.
\end{cor}

\begin{proof}
We first prove the statement in the solvable case when $\Gamma$ is a
lattice of $G$.  If $J$ is an invariant \acs, then its first Chern
class vanishes. This implies that the Ricci form $\rho$ associated to
any $J$-compatible metric $g$ is exact. If $g$ is invariant, then also
$\rho$ is left-invariant and the scalar curvature
$s=\mathrm{tr}_g\rho$ is constant. Therefore if $g$ is an invariant
$J$ compatible \aK metric with associated symplectic form $\omega$ we
have
$$ 0=\int_M \rho\we \omega=c\,s\!\int_M \omega^2
$$ $c$ being a non-zero constant. Hence $s$ vanishes and the claim
follows.

On the other hand, if $M$ is an infra-solvmanifold with an invariant
\aK structure, then it is covered by an \aK solvmanifold. Since the
tensor $\mathcal{R}$ depends only from the \aK structure, then
condition \eqref{yau} is invariant for covering which preserve the \aK
structure and the claim follows.
\end{proof}

\section{Manifolds modelled on $\Nil^3\times\R$}\label{Kodaira}

In this section, we consider in more detail the Kodaira--Thurston
manifold $M= \Gamma_0\backslash X$, a discrete  quotient of
$X=\Nil^3\times\R$. The standard choice of $\Gamma_0$ (case (1\1) of
Section \ref{classification} with $\la=1$) allows us to consider the
two principal $T^2$-fibrations \eqref{pi's} with 2-torus base.

 The manifold $M$ has a global basis consisting of real 1-forms
\begin{equation}\label{basisKT}
e^1=dy,\qquad e^2=dx,\qquad e^3=dt,\qquad e^4=dz-x\kern1pt dy.
\end{equation}
This basis conforms to the first structure equation in \eqref{str},
although the 1-dimensional factor in $\Nil^3\times\R$ is in
third place. Any tensor that can be expressed in terms of it with
constant coefficients is called \emph{invariant} to reflect the fact
that it is induced from a left-invariant tensor on the Lie group $X$.
In particular, $M$ inherits a complex structure and holomorphic
symplectic form
\begin{equation}\label{Theta0}
(e^1+ie^2)\we(e^3+ie^4)=d\zeta_1\we d\zeta_2,
\end{equation}
where \[\textstyle \zeta_1=y+ix,\qquad\zeta_2=t-\frac12x^2+iz\] are
local complex coordinates on the base and fibres. The latter become
bilagrangian, and $\pi_{xy}$ is holomorphic. This observation
is readily generalized.

We first describe the space of invariant symplectic forms on $M$. An
element of this space lies in
\[\textstyle\ker d=\left<e^{12},e^{13},e^{14},e^{23},e^{24}\right>,\]
and so equals
\[\Om_{\la,A}=\la e^{12} + e^1\we(ae^3+be^4)-e^2\we(ce^3+de^4),\]
for some $\la\in\R$ and some real $2\times2$ matrix
$A=\hbox{\small$\left(\!\!\begin{array}{cc}a&b\\c&d\end{array}\!\!\right)$}$
with $\det A\ne0$. With this convention, \eqref{Theta0} equals
$\Om_{0,I}+i\,\Om_{0,J}$ where $I$ is the identity matrix and
$J=\hbox{\small$\left(\!\!\begin{array}{cc}0&1\\\!-1&0\end{array}\!\!\right)$}$.

\begin{lemma}\label{lagf}
The fibration $\pi_{xy}$ is Lagrangian with respect to any invariant
symplectic form on $M$.
\end{lemma}

\begin{proof} The simple 2-form $e^{12}$ represents a non-zero volume form
on the base of $\pi_{xy}$. At the level of differential forms, the
restriction of $\Om_{\la,A}$ to any fibre is therefore determined
by the wedge product $\Om_{\la,A}\we e^{12}=0$. The result
follows.
\end{proof}

More generally, consider two such forms $\Om_2,\Om_3$ with
$\Om_i=\Om_{\la_i,A_i}$. It is not difficult to arrange for
$\Th=\Om_2+i\,\Om_3$ to satisfy $\Th\we\Th=0$; this is equivalent to
asserting
\begin{equation}\label{detB}
\det(B)=1,\qquad \mathrm{tr}(B)=0,\qquad B=A_2^{-1}A_3.
\end{equation}
When conditions \eqref{detB} hold, we can write
$\Th=\alpha_1\we\alpha_2$ over $\C$, whence
$\Lambda^{1,0}=\left<\alpha_1,\,\alpha_2\right>$ is the space of
$(1,0)$-forms relative to some invariant complex structure $J_1$ on
$M$. Observe that
\[ 0=d\Th=d\alpha_1\we\alpha_2-\alpha_1\we d\alpha_2,\]
which proves that $(d\alpha_i)^{0,2}=0$ for $i=1,2$. The integrability
of $J_1$ then follows from the Newlander--Nirenberg theorem.

Since $\Th\we e^{12}=0$, the real 2-form $e^{12}$ has type $(1,0)$
relative to $J_1$. It follows that we can choose $p,q\in\C$ so that
$\alpha_1=pe^1+qe^2$, defining a complex structure on the base
$\bT_{xy}$ relative to which $\pi_{xy}$ is holomorphic. The fibres of
$\pi_{xy}$ themselves become holomorphic curves in $(M,J_1)$. We
complete these observations with

\begin{prop}
Given an invariant \aK structure $(g,J_2,\Om_2)$ on
$M$, we can choose a holomorphic symplectic structure
$\Th=\Om_2+i\,\Om_3$ and associated complex structure $J_1$
for which $\pi_{xy}$ is holomorphic. Conversely, any invariant complex
structure on $M$ arises in this way.
\end{prop}

\begin{proof}
By assumption, $g$ is a scalar product for which
\[ g(J_2X,J_2Y)=g(X,Y),\qquad g(J_2X,Y)=\Om_2(X,Y).\]
Orient $M$ by means of $J_2$, so that the 3-dimensional space $\La^+$
of self-dual 2-forms relative to $g$ contains the closed 2-form
$\Om_2$. Since $\ker d$ has dimension 5 on the space of 2-forms, we
can find a non-zero closed 2-form $\Om_3$ in the orthogonal
complement of $\Om_2$ in $\La^+$. The orthogonality means that
$\Om_2\we\Om_3=0$ and we can normalize $\Om_3$ so that
$\Om_2\we\Om_2=\Om_3\we\Om_3$. It follows that
$\Th=\Om_2+i\,\Om_3$ satisfies $\Th\we\Th=0$ as in
\eqref{detB}, and $\pi_{xy}$ is holomorphic.

Conversely, suppose that $J$ is an arbitrary invariant complex
structure. Since $J$ is integrable, its space $\La^{1,0}$ of
$(1,0)$-forms is generated by invariant 1-forms $\alpha_1,\alpha_2$
for which $d\alpha_1=0$ and $d\alpha_2$ is divisible by $\alpha_1$, by
\cite[Theorem 1.3]{Sal}. It follows that the invariant 2-form
$\alpha_1\we\alpha_2\in \Lambda^{2,0}$ is closed and, after fixing
a compatible metric, we have recovered the initial hypothesis.
\end{proof}

\begin{rem} \rm The discussion above shows that $M$ possesses a
  non-trivial moduli space of holomorphic symplectic structures $\Th$.
  Specifying the invariant complex structure $J_1$ is equivalent to
  specifying the projective class $\left<\Th\right>$, and $J_1$ in
  turn determines a complex structure on both the base $\bT_{xy}^2$
  and a fibre $T^2$ that become elliptic curves. Taken together, these
  two complex structures can be specified by a pair of invariants
  $(j_1,j_2)\in\C^2$. A third parameter is required to `scale' base
  and fibre, but can be eliminated by quotienting by the action of
  right translation by $\Nil^3$ commuting with $\Gamma$, cf.\
  \cite{KS,Poon}.
\end{rem}

Given an invariant \aK structure, extended as in the
  proposition, it is possible to set $\Th=\alpha_1\we\alpha_2$,
  where
\[\textstyle \alpha_1 = f^1 + if^2\in\left<e^1,e^2\right>_\C,
\qquad\alpha_2= f^3+if^4,\]
the $f^i$ have equal norms relative to $g$, and $df^i=0$ for
$i=1,2,3$. (To achieve the latter, multiply $\alpha_2$ by a unit complex
number so that $f^3$ is a linear combination of $e^1,e^2,e^3$.) We can
now write $df^4=k f^{12}$ with $k\ne0$. Apart from this last factor,
we have recovered the structure equation \eqref{str}(1). In
particular, any two invariant holomorphic symplectic structures are
equivalent by a Lie algebra automorphism.

With the assumptions of the last paragraph, the space $\Lambda^+$ of
self-dual 2-forms (see \eqref{ext2}) is generated by
\begin{equation}\label{selfdual}
\Om_1=f^{12}+f^{34},\qquad\Om_2=f^{13}+f^{42},\qquad
\Om_3=f^{14}+f^{23}.
\end{equation}
The space $\La^+$ determines the conformal (or rather, in an invariant
context, the homothety) class of $g=\sum f^i\otimes f^i$. Its
annihilator 
$\La^- = \langle \Om_1,\Om_2,\Om_3\rangle^\circ$
(relative to wedge product) is the space of anti-self-dual $2$-forms.

We shall in fact fix the homothety class of $g$ by requiring that
\begin{equation}\label{Theta}
\Th=(f^1+if^2)\we(f^3+if^4)=\Om_2+i\,\Om_3
\end{equation}
have the same norm as \eqref{Theta0}. The forms \eqref{selfdual} are
therefore characterized by the conditions
\begin{equation}\label{conditions}
\left\{
\begin{array}{l}
\Om_i\we\Om_j=0,\quad i\ne j;\\[5pt]
\Om_i\we\Om_i=2e^{1234},\quad i=1,2,3;\\[5pt]
d\Om_i=0,\quad i=2,3.
\end{array}
\right.
\end{equation}
Setting $\Om_i(v,w)=g(J_iv,w)$ extends the definition on $J_1$ into a
triple $J_1,J_2,J_3$ of \acs s.

The main result of this section is

\begin{theorem}\label{KT} Let $F\in C^\infty(\bT)$
  be a smooth function on the base $\bT=\bT_{xy}^2$ such that
\begin{equation}\label{second}
   \int_\bT(\eF-1)=0.
\end{equation}
Fix a triple of $2$-forms \eqref{selfdual} satisfying
\eqref{conditions}, and set $\Th=\Om_2+i\Om_3$. Then there exists a
closed complex-valued $2$-form $\tTh$ on $M$ such that
\begin{itemize}
\vskip2pt\item \ $\tTh\we\tTh=0$ and
$\tTh\we\overline{\tTh}=4\eF e^{1234}$;
\vskip2pt\item \ $\tTh=h\,\Th+\Phi_-$ where $h\in C^\infty(\bT)$
and $\Phi_-\colon\bT\to(\La^2_-)_\C$;
\vskip2pt\item \ $[\tTh]=[\Th]$ in $H^2(M,\C)$.
\end{itemize}
\end{theorem}

\begin{proof} For simplicity, we first suppose that $f^1=e^1=dx$ and
  $f^2=e^2=dy$, which is the case when $\Th$ has the standard form
  \eqref{Theta0}. Consider the 2-dimensional Monge--Amp\`ere
  equation
  \begin{equation}\label{third}
 P_{xx}P_{yy} - P_{xy}^2 = \eF,
\end{equation}
where
\begin{equation}\label{Pp}
\textstyle P(x,y)=\frac12(x^2+y^2)+p(x,y),
\end{equation}
with $p\in C^\infty(\bT)$.  This equation can be re-written
\begin{equation}\label{dPu}
dP_x\we dP_y=\eF\,f^{12}.
\end{equation}
On the other hand, substituting \eqref{Pp}, 
$\int_\bT dP_x\we dP_y=\int_\bT f^{12}$, since $p$ and its derivatives integrate to
zero. Thus, \eqref{second} is a necessary condition for a solution to
\eqref{dPu}.

We define
\begin{equation}\label{tT}
\tTh = (dP_x+i\,dP_y)\we(f^3+if^4),
\end{equation}
in analogy to \eqref{Theta}. The first conditions
\[\begin{array}{c}
\tTh\we\tTh = 0,\\[5pt]
\tTh\we\overline{\tTh}=4\,dP_x\we dP_y\we f^{34} = 4\eF\,e^{1234},
\end{array}
\]
are then immediate consequences of \eqref{tT} and \eqref{dPu}.

Recalling that $\Om_1=f^{12}+f^{34}$, it is equally obvious that
$\tTh\we\Om_1=0$. We also have $\tTh\we\Th=0$, and, with a bit more
work,
\[\tTh\we\overline\Th=2(P_{xx}+P_{yy}).\]
The second bullet point follows by setting
$h=\frac12(P_{xx}+P_{yy})$.

Finally, we may write
\begin{equation}\label{dalpha}
\begin{array}{rcl} \tTh - \Th
&=& (dp_x+idp_y)\we(f^3+if^4)\\[5pt]
&=& d\left[(p_x+ip_y)\we(f^3+if^4)\right]
      -(p_x+ip_y)\we ikf^{12}\\[5pt]
&=& d\alpha,
\end{array}
\end{equation}
where
\[\alpha = -kp(f^1+if^2) + (p_x+ip_y)(f^3+if^4)\]
is a well-defined 1-form on $\bT$. Note that we have used the
assumption $df^4=kf^{12}$ reflecting the arbitrary invariant metric at
the start.\smallbreak

To know that we can solve \eqref{third}, it suffices to take $n=2$,
$S_{ij}=\delta_{ij}$ and $G=F$ in the following theorem, taken from
\cite{Li}.\smallbreak

{\it Let $\bT^n=\R^n/\Z^n$ be a torus with flat coordinates
  $x_i$. Suppose that $F\in C^{\infty}(\bT^n)$ is a positive function,
  and $(S_{ij})$ a symmetric positive-definite $n\times n$ matrix satisfying
\begin{equation}\label{CF}
\int_{\bT^n}\!\big(F-\det(S_{ij})\big) = 0.
\end{equation}
Consider the equation
\begin{equation}\label{detC}
 \det\left(S_{ij} + \frac{\pd^2 p}{\pd x_i\pd x_j}\right) 
= \eE^G \qquad\hbox{on}\quad\bT^n,
\end{equation}
together with the positive-definite condition
\begin{equation}\label{epd}
\left(\!S_{ij} + \frac{\pd^2 p}{\pd x_i\pd x_j}\right) > 0
\qquad\hbox{on}\quad\bT^n.
\end{equation}
Then there exists $p\in C^{\infty}(\bT^n)$ satisfying \eqref{detC} and
\eqref{epd}, and $p$ is unique up to addition of a constant. Moreover,
\eqref{CF} is necessary for the solvability of
\eqref{detC}.}\smallbreak

In general, there will be a constant matrix $A=(A_{ij})$ such that
$e^i = \sum_{i=1}^2 A_{ij}f^j$. We therefore make a linear substitution
of coordinates
\[ x=A_{11}u+A_{12}v,\quad y=A_{21}u+A_{22}v,\]
so that \[ f^1=du,\qquad f^2=dv.\] 
This allows us to proceed as before; the differential
equations to be solved are the same except that
all derivatives with respect to $x,y$ are replaced by those with 
respect to $u,v$. Since
\begin{equation}\label{BA} H = 
\left(\!\begin{array}{cc}
P\sb{uu}&P\sb{uv}\\[3pt] P\sb{uv}&P\sb{vv}\end{array}\!\right) 
= I + A\tp\,(P\sb{x_ix_j})A,
\end{equation}
the Monge--Amp\`ere equation $\det H=\eF$ in the new coordinates takes
the form \eqref{detC} with
\[ (S_{ij})=(AA\tp)^{-1},\qquad f=\eF/(\det A)^2.\] 
The theorem then implies the existence
$p\in \C^{\infty}(\bT^2)$ solving \eqref{second}. \end{proof}

Taking the real part of the second bullet point yields
\[ \widetilde\Om_2=h\Om_2+\Phi_2,\]
where $\Phi_2$ is a $(1,1)$-form relative to all of $J_1,J_2,J_3$.
In particular,
\[ d(\mathrm{Re}\,\alpha) =
\widetilde\Om_2-\Om_2=(h-1)\Om_2+\Phi_2\] is a $(1,1)$-form
relative to $J_2$. Moreover, the new structure has the assigned volume
form
\[ \widetilde\Om_2\we\widetilde\Om_2
= \eF \Om_2\we\Om_2.\]
In conclusion,

\begin{cor} \label{cor45} Given an invariant \aK structure
  $(g,J_2,\Om_2)$ on $M$, the associated Calabi--Yau problem admits a
  unique solution.
\end{cor}

\begin{rem}\label{!!}\rm The uniqueness statement for solutions of
Monge--Amp\`ere on the 2-torus base will translate into a uniqueness
result for the 4-dimensional problem. The only gap in our presentation
is the proof that the problem on $M$ necessarily gives rise to a
solution of \eqref{third}. However, uniqueness is already covered by
Proposition~\ref{!}, whose proof translates directly into the
following neat argument for \eqref{third}.

Taking \[2h=P\sb{uu}+P\sb{vv} = 4 + \Delta p,\] to be positive, the
triple $(P\sb{uu},P\sb{vv},P\sb{uv})$ defines (at each point of $\bT$)
a vector inside the forward light cone in $\R^{1,2}$ with a Lorentzian
norm equal to $\eF$. If $P,\widetilde P$ are two solutions, set
\[ P'=P-\widetilde P,\qquad P''=P+\widetilde P.\]
Then $(P'\sb{uu},P'\sb{vv},P'\sb{uv})$ (having zero Lorentzian product
with $(P''\sb{uu},P''\sb{vv},P''\sb{uv})$) is everywhere spacelike or
zero, so $dP'_u\we dP'_v$ is a non-negative multiple of $e^{12}$. But
$P'$ is biperiodic, so the integral of this 2-form over $\bT^2$
vanishes and $P'$ must be constant.
\end{rem}

\subsection{The potential}
To further understand the proof of Theorem~\ref{KT}, note that $\tTh$,
by its very definition \eqref{tT}, is holomorphic relative to the
complex structure with local coordinate $P_u+iP_v$ on the base and
equal to $J_1$ on the fibres. In this sense, $P$ is a potential
function for the modified complex structure, and the assigned volume
form on $\bT^2$ then determines the \emph{Hessian metric}
\begin{equation}\label{ds2}
ds^2=P\sb{uu}du^2+2P\sb{uv}du\kern1pt dv+P\sb{vv}dv^2,
\end{equation}
represented by the matrix $H$ of \eqref{BA}. This fundamental form is
then duplicated on the fibres of $M$. Such metrics are well known in
work on special Lagrangian and toric geometry in a K\"ahler setting
\cite{Hitchin,AS}.

The potential $P$ does not feature in \cite{TW}, but is an integral of
functions considered there. Nonetheless, the Tosatti--Weinkove
estimate
\begin{equation}\label{ExtimateTosWein}
\tilde\Delta h\ \ge\ \inf_{\bT^2}\Delta F
\end{equation}
makes sense in the above context, where $\tilde\Delta$ is the
Laplacian relative to \eqref{ds2}, and their proof of it translates
into the following. Take derivatives of \eqref{third}, and rearrange
them to give an equation of the form
\[\Delta F = \Delta_\mathrm{even}-\Delta_\mathrm{odd},\] where
$\Delta_\mathrm{even}$ (respectively $\Delta_\mathrm{odd}$) involves
products of partial derivatives of even (respectively odd) total
order. Moreover,
\[\begin{array}{l}
\Delta_\mathrm{even} =
\hbox{tr}\left(H^{-1}\mathrm{Hess}\,h\right) \tilde\Delta h,\\[5pt]
\Delta_\mathrm{odd} = \wp(H^{-1}H_u)+\wp(H^{-1}H_v),
\end{array}\]
$\wp(B)=(\hbox{tr}\,B)^2-2\det B\ge0$ for $B$ a symmetric matrix.

\subsection{Other manifolds belonging to the family (1\1).}
All the computations we have performed in this section on the
Kodaira--Thurston manifold can be used for every manifold
$M=\Gamma\backslash (Nil^3\times \R)$ belonging to class (1\1). Every
$\Gamma$ is a lattice of $Nil^3\times \R$, and the quotient $M$ is a
genuine nilmanifold.

\section{Manifolds modelled on $\Nil^4$}\label{3-step}

Let $M$ be the total space of a torus bundle belonging to case (2) of
Section~\ref{classification}, so that $M$ is a nilmanifold associated
to the $3$-step nilpotent Lie group \eqref{consist}. The third row of
\eqref{mappings} yields a projection
\begin{equation}\label{pizt}
\pi_{zt}\colon M\to \bT^2_{zt}
\end{equation}
with $T^2$ fibres. We adopt the following basis of left-invariant
vector fields:
\[\textstyle
e_1=-\pd_t,\qquad
e_2=\pd_y + t\pd_x,\qquad
e_3=\pd_z + t\pd_y + \frac12 t^2\pd_x, \qquad
e_4=\pd_x,
\]
and dual basis of 1-forms:
\begin{equation}\label{-tyz}\textstyle
e^1=-dt,\qquad
e^2=dy - t\kern1pt dz,\qquad
e^3=dz,\qquad
e^4=dx-t\,dy+\frac12t^2dz.
\end{equation}
With this convention, $(e^i)$ satisfies \eqref{3st}, which coincides
with setting $\la=1$ in \eqref{def} in the next lemma.

The analogue of Lemma~\ref{lagf} is valid for $\pi_{zt}$. This is
because any symplectic form $\Om$ belongs to
\[\textstyle \ker d=\left<e^{12},e^{13},e^{14},e^{23}\right>.\]
This time it is $e^{13}=dz\we dt$ that represents a volume form on
the base, and
\begin{equation}\label{Om13}
\Om\we e^{13}=0.
\end{equation}
Of course, it remains true that $\Om\we e^{12}=0$.

The canonical choice for a positively-oriented symplectic form is
$e^{14}+e^{23}$, and the next result shows that the general case is
not so very different. We state it in a form that will also serve in
Section~\ref{2fib}.

\begin{lemma}\label{filter}
  Let $(g,J,\Om)$ be an invariant \aK structure on a $4$-dimensional
  nilmanifold that satisfies
\begin{equation}\label{def}
de^1=0,\qquad de^2=\la e^{13},\qquad de^3=0,\qquad de^4=e^{12},
\end{equation}
with $\la\in\R$. Assume \eqref{Om13} (which is an extra hypothesis
only if $\la=0$). There exists an orthonormal basis $(f^i)$ for which
\begin{equation}\label{Omega3}
\Om=f^{14}+f^{23},
\end{equation}
and
\[\textstyle f^1\in\left<e^1\right>,\qquad f^2\in\left<e^1,e^3\right>,
\qquad f^3\in\left<e^1,e^3,e^2\right>.\]
\end{lemma}

\smallbreak\noindent \emph{NB.} The inherent filtration would be more
memorable had we interchanged $e^2\leftrightarrow e^3$, but this would
have caused bigger notational difficulties elsewhere.

\begin{proof}
We can certainly find an orthonormal basis $(f^i)$ of 1-forms for
which \eqref{Omega3} is valid, whilst retaining the freedom to act by
the stabilizer $U(2)$. Using the latter, we firstly ensure that
$f^1=ke^1$ for some $k\ne0$, leaving freedom to rotate in the $f^{23}$
plane. Any 2-form in the image of $d$ is divisible by $e^1$, so
$d(f^{14})=-f^1\we df^4=0$.

Since \eqref{Om13} holds in all cases, $f^{23}\we e^{13}=0$.  This
means that we can choose $f^2$ to be a linear combination of
$e^1,e^3$. We now have
\[0 = d(f^{23}) = -f^2\we df^3 = e^{13}\we\sigma,\]
for some 1-form $\sigma\in\left<e^2,\la e^3\right>$. But this forces
$\sigma\in\left<\la e^3\right>$ and $f^3$ cannot have a component in
$e^4$.
\end{proof}

The ingredients are now in place to prove another special case of
Theorem~\ref{main}.

\begin{theorem}\label{3stepthm}
Let $M$ be a discrete quotient of $\Nil^4$ admitting a fibration
\eqref{pizt} over $\bT=\bT^2_{zt}$. Given any invariant \aK structure
$(g,J,\Om)$ on $M$, and a volume form $\sigma=\eF\,\Om^2$ with $F\in
C^\infty(\bT)$ such that $\int_\bT(\eF-1)=0$, the problem
\eqref{CY2}, \eqref{CY3} admits a unique solution.
\end{theorem}

\begin{proof}
Let $(f^i)$ be a coframe as in Lemma~\ref{filter}. Then we can write
\begin{equation}\label{ACB}
  dt = Af^1,\qquad dz = Cf^1+Bf^2
\end{equation}
for some $A,C\in \R$ with $A>0$ and $B\ne0$. We also have
\begin{equation}\label{df3df4}
df^3=kf^{12},\qquad df^4=l f^{12}+mf^{13},
\end{equation}
with $k,m$ non-zero. Consider now a $1$-form
\[ \alpha = \sum_{i=1}^4 a_if^i,\]
whose coefficients $a_i=a_i(t,z)$ are functions on the base. We have
\[\begin{array}{rcl} 
d\alpha 
&=& a_3kf^{12}+a_4(l f^{12}+mf^{13})+\sum\limits_{i=1}^4da_i\we f^i
\\[10pt]
&=& (ka_3+l a_4+Aa_{2,t}-Ba_{1,z})f^{12} + (ma_4+Aa_{3,t}+Ca_{3,z})f^{13}
\\[5pt]
&& \hskip50pt+\,(Aa_{4,t}+Ca_{4,z})f^{14}+Ba_{3,z}f^{23}+Ba_{4,z}f^{24}.
\end{array}\]
To ensure that $d\alpha$ has type $(1,1)$ relative to $J$, we need
\begin{equation}\label{klm}
\begin{array}{c}
ka_3+l a_4+Aa_{2,t}-Ba_{1,z}=0,\\[5pt]
ma_4+Aa_{3,t}+Ca_{3,z}-Ba_{4,z}=0,
\end{array}
\end{equation}
so that
\[\Om+d\alpha = (1+Aa_{4,t}+Ca_{4,z})f^{14} + (1+Ba_{3,z})f^{23} + 
Ba_{4,z}(f^{13}-f^{42})\] belongs to $\left<\Om\right>+\La^-$, as in
Theorem~\ref{KT}. The volume constraint \eqref{CY2} is now
\begin{equation}\label{CY4}
 \left(1+Aa_{4,t}+Ca_{4,z}\right)\left(1+Ba_{3,z}\right)-(Ba_{4,z})^2 = \eF,
\end{equation}
where $F=F(t,z)$ is a function on the base. 

In view of Proposition~\ref{!}, we need only produce one solution. For
this purpose, define
\[a_3 = \frac BA p_z - \frac mA p,\qquad a_4 = p_t + \frac CA p_z.\]
Then \[\begin{aligned}
& A(ma_4+Aa_{3,t}+Ca_{3,z}-Ba_{4,z})\\
& \hskip20pt
= m(Ap_t + Cp_z)+A(Bp_{zt}-mp_t)+C(Bp_{zz}-mp_z)-B(Ap_{zt} + Cp_{zz})\\
& \hskip20pt
= 0,
\end{aligned}\]
giving the second equation in \eqref{klm}. After a long computation, 
\eqref{CY4} becomes \begin{equation}\label{GMAeq0}
\hfill
\Big(\frac{B^2+C^2}{AB^2} + p_{tt} - \frac{mC^2}{A^2\!B}p_z\Big) 
\Big(\frac A{B^2} + p_{zz} - \frac mB p_z\Big) \! -
\Big(\!\!-\!\frac C{B^2} + p_{zt} + \frac{mC}{AB}p_z\Big)^2 
\! = \frac1{B^2}\eF\!\!,
\end{equation}
Equations of this type are discussed in the next section. The
necessary normalization of $\eF$ is part of our hypothesis, and
Theorem~\ref{th1} will guarantee a solution unique up to the addition
of a constant.

It remains to satisfy the first equation in \eqref{klm}, which
becomes
\[ kBp_z- mkp+Al p_t + l Cp_z+A^2a_{2,t}-ABa_{1,z}=0.\]
So far, $a_1$ and $a_2$ are unconstrained, but they must end up
biperiodic. So it suffices to solve the simpler equation
\begin{equation}\label{abq}
(\ta_2)_t-(\ta_1)_z = q,
\end{equation}
where $\ta_1,\ta_2$ and $q=-kmp$ are biperiodic, and then add or
subtract multiplies of $p$ from $\ta_1,\ta_2$ to get $a_1,a_2$
respectively. Noting that \eqref{abq} translates into
\[ d(\ta_1\,dt+\ta_2\,dz)=q\,dt\we dz,\]
we first adjust $q$ by a constant so that
$\int_\bT q\,f^{12}=0$. This means that $[q\,dt\we dz]$ vanishes in
$H^2(M,\R)$, so the 2-form is exact and \eqref{abq} has a solution
with $\ta_1,\ta_2\in C^\infty(\bT)$.
\end{proof}

Just as in Section~\ref{Kodaira}, the solution $d\alpha$ of the
Calabi--Yau problem involves a well-identifiable component on the base
as in \eqref{dalpha}, arising from \eqref{abq}.

\section{The generalized Monge--Amp\`ere equation}\label{GMA}

In this section, we shall work exclusively with functions defined on a
2-torus $\bT=\bT^2$. More precisely, we write $p\in C^\infty(\bT)$ to
mean that $p\colon \R^2\to\R$ is a smooth function satisfying
\begin{equation}\label{biperiod} p(x+1,y)=p(x,y)=p(x,y+1).
\end{equation}
Since we have already dealt with the fibration over $\bT^2_{xy}$ in
Section~\ref{Kodaira}, the pair of coordinates $(x,y)$ will in
applications take on the role of either $(t,z)$ (for the previous
section) or $(y,t)$ (for the next section).

Given a function $F\in C^\infty(\bT)$, fix a real positive-definite
symmetric matrix
$$
\left(\begin{array}{cc}
a&c\\
c&b
\end{array}
\right)>0,
$$ and two $2\times 2$ real matrices $(l_{ij})$ and $(m_{ij})$
satisfying
\begin{equation}\label{lm}
\begin{array}{l}
m_{11}l_{22}=0,\\
l_{11}l_{22}-l_{12}^2=0\\
m_{11}m_{22}-m_{12}^2=0\\
l_{11}m_{22}+l_{22}m_{11}-2l_{12}m_{12}=0.
\end{array}
\end{equation}
Consider the partial differential equation
\begin{equation}\label{GMAeq}
 \hskip5pt
(a+p_{xx}-l_{11}p_x-m_{11}p_y)(b+p_{yy}-l_{22}p_x-m_{22}p_y)
-(c+p_{xy}-l_{12}p_x-m_{12}p_y)^2=\eF,
\end{equation}
for $p\in C^\infty(\bT)$ satisfying
\begin{equation}
\left(\begin{array}{cc}
a+p_{xx}-l_{11}p_x-m_{11}p_y & c+p_{xy}-l_{12}p_x-m_{12}p_y\\[2pt]
c+ p_{xy} -l_{12}p_x-m_{12}p_y & b+p_{yy}-l_{22}p_x-m_{22}p_y
\end{array}\right)>0
\label{16new}
\end{equation}
at all points on the 2-torus $\bT$.

\begin{ex}
After the change of coordinates from $(x,y)$ to $(t,z)$, the equation
\eqref{GMAeq0} that arose in the previous section is obtained by
substituing
\[ a=\frac A{B^2},\quad b=\frac{B^2+C^2}{AB^2},\quad c=-\frac C{B^2};
\qquad m_{11}=\frac mB,\quad m_{22}=\frac{m C^2}{A^2B},\quad
m_{12}=-\frac{mC}{AB},\] 
and $(l_{ij})=0$. Conditions \ref{lm}
and \ref{16new} are verified. The particular case
\[ a=1,\quad b=1,\quad c=0;\qquad m_{11}=1,\quad m_{22}=m_{12}=0\]
occurs for the standard choice $\Om=e^{13}+e^{42}$ of symplectic form
in Theorem~\ref{3stepthm}. The matrix $(l_{ij})$ would intervene if
the linear transformation \eqref{ACB} were no longer triangular. We
have opted for the general form \eqref{GMAeq} so as to emphasize the
invariance of our problem under an arbitrary linear change of
coordinates on the 2-torus base.
\end{ex}

The main result of this section is

\begin{theorem}\label{th1}
Let $a,b,c,(l_{ij}),(m_{ij})$ and $F$ be as above. Then
{\rm(\ref{GMAeq})} and {\rm(\ref{16new})} have a solution $p\in
C^2(\bT)$ if and only if
\begin{equation}
\int_\bT(\eF-ab+c^2)=0.
\label{17newnew}
\end{equation}
Moreover, $p$ is unique modulo addition of a constant, and any $C^2$
solution is in $C^\infty$.
\label{thm1}
\end{theorem}

\noindent\textit{Proof.} We first prove that (\ref{17newnew}) is a
necessary condition for (\ref{GMAeq}) to have a solution. Using the
periodicity of $p$, we see that the following functions integrate to
zero over the 2-torus:
\[\begin{array}{c}
p_x,\quad p_y,\quad p_{xx},\quad p_{xy},\quad p_{yy},\\[5pt]
p_{xx}p_x= \frac12\pd_x(p_x^2),\quad p_{yy}p_y,\quad
p_{xy}p_x= \frac12\pd_y(p_x^2),\quad p_{xy}p_y.
\end{array}\]
Integrating by parts to obtain the last function, we also have
\[\int_\bT p_{yy}p_x=-\!\int_\bT p_yp_{xy}=0,\qquad \int_\bT p_{xx}p_y = 0.\]
If \eqref{GMAeq} has a solution $p\in C^2(\bT)$ then, from above and 
assumptions on $(l_{ij}),(m_{ij})$,
\[\begin{array}{rl}\displaystyle
\int_\bT\eF\!\!
&=\displaystyle\ 
\int_\bT\big[(a+p_{xx})(b+p_{yy})-(c+p_{xy})^2\\
&\hskip50pt\displaystyle\ 
+(l_{11}p_x+m_{11}p_y)(l_{22}p_x+m_{22}p_y)
-(l_{12}p_x+m_{12}p_y)^2\big]\\[10pt]
&= \displaystyle\
(ab-c^2)|\bT|+ \int_\bT\left[p_{yy}p_{xx}-p_{xy}^2\right]\\
&\hskip20pt\displaystyle
+\!\int_\bT\Big[ p_x^2(l_{11}l_{22}-l_{12}^2)+p_y^2(m_{11}m_{22}-m_{12}^2)
+p_xp_y(l_{11}m_{22}+l_{22}m_{11}-2l_{12}m_{12})\Big]\\[15pt]
&=\displaystyle\
(ab-c^2)|\bT|+ \int_\bT\left[p_{yy}p_{xx}-p_{xy}^2\right].
\end{array}\]

Similarly,
\begin{equation}\label{intparts}
\int_\bT\left[ p_{yy}p_{xx}-p_{xy}^2\right]
=\int_\bT\left[ -p_y p_{xxy} -p_{xy}^2\right]
=0.
\end{equation}
It follows that \eqref{17newnew} holds.

Now we prove the uniqueness of solutions modulo addition of constants.
Let $p$ and $\widetilde p$ be two solutions of (\ref{GMAeq}) and
(\ref{16new}).  Let $h$ be a constant such that
$$
\min_\bT\big[\,(p+h)-\widetilde p\,\big]=0.
$$
Clearly $p+h$ is also a solution of
(\ref{GMAeq}) and (\ref{16new}).
By the mean value theorem, we have
\begin{eqnarray*}
0\!\!&=&\!\!
a_{11}(x,y) \pd_{yy}
\big[ (p+h)-\widetilde p\,\big]
+2a_{12}(x,y) \pd_{xy}
\big[(p+h)-\widetilde p\,\big]
+ a_{22} (x,y) \pd_{yy}
\big[(p+h)-\widetilde p\,\big]
\\&&\hskip50pt
+b(x,y) \pd_y\big[(p+h)-\widetilde p\,\big],
\end{eqnarray*}
where $(a_{ij}(x,y))$ is a positive-definite symmetric matrix
function, which like $b(x,y)$, is continuous on $\bT$. Thus, by the
strong maximum principle, $(p+h)-\widetilde p\equiv0$. The uniqueness
result follows.

Since the equation is elliptic, and all data is smooth, a $C^2$
solution $x$ is in $C^\infty$.

To prove the existence part of the theorem, we use the method of continuity.
This requires a priori estimates which we derive below.

\begin{lemma}
Let $f\in C^2(\R)$ be a $1$-periodic function satisfying, for some
$\alpha, \beta\in R$,
$$
f''(s)+\alpha f'(s)\ge \beta,\qquad -\infty<s<\infty.
$$
Then
$$
|f'(s)|\le
2|\beta|e^{ 2|\alpha|} ,\qquad\forall\ s\in\R.
$$
\label{lem1}
\end{lemma}

\begin{proof}
We know that
$$
\frac d{ds}\big [\eE^{\alpha s} f'(s)\big]\ge \beta \eE^{as}.
$$
Since $f$ is $1$-periodic, there exists $\bar s\in [0, 1)$
such that $f'(\bar s)=0$.
It follows that
$$
 \eE^{\alpha s} f'(s)\ge \beta\!\int_{ \bar s }^s\eE^{\alpha t}\,dt, 
\qquad s\ge 1.
$$
This implies
$$
 f'(s)\ge
 \eE^{-\alpha s} \beta\!\int_{ \bar s }^s\eE^{\alpha t}\,dt
\ge -2|\beta|\eE^{ 2|\alpha|},\qquad 1\le s\le 2.
$$
Similarly,
$$
-\eE^{as}
f'(s)\ge \beta\!\int_s^{ \bar s}\eE^{\alpha t}\,dt,\qquad s\le 0,
$$
and
$$
-f'(s)\ge \beta \eE^{-\alpha s} \int_s^{ \bar s}
\eE^{\alpha t}\,dt
\ge -2|\beta|\eE^{ 2|\alpha|},\qquad\quad  -1\le s\le 0.
$$ The desired estimate follows from the above and the $1$-periodicity
of $f$.
\end{proof}

For $0\le t\le 1$, consider the equation
\begin{eqnarray}
&&\kern-10pt
(a+p_{xx}-tl_{11}p_x-tm_{11}p_y)
(b+p_{yy}-tl_{22}p_x-tm_{22}p_y) -(c+p_{xy}-tl_{12}p_x-tm_{12}p_y)^2
\nonumber\\
&&\hskip100pt = t\eF+(1-t)(ab-c^2)
\qquad\mbox{on}\quad\bT,
\label{15newt}
\end{eqnarray}
together with the condition
\begin{equation}
 \left(\begin{array}{cc}
a+p_{xx}-tl_{11}p_x-tm_{11}p_y & c+p_{xy}-tl_{12}p_x-tm_{12}p_y\\[2pt]
c+p_{xy}-tl_{12}p_x-tm_{12}p_y& b+p_{yy}-tl_{22}p_x-tm_{22}p_y
\end{array}\right)
>0
\qquad \mbox{on}\quad\bT.
\label{16newt}
\end{equation}
For this problem, we have

\begin{lemma} For $0\le t\le 1$, let $p\in C^2(\bT)$ be a solution of
{\rm(\ref{15newt})} and {\rm(\ref{16newt})}.  Then
\begin{equation}
|\nabla p|\le C \qquad\mbox{on}\quad\bT,
\label{18}
\end{equation}
and
\begin{equation}
\bigg|\>p- \frac 1{|\bT|}\!\int_{  \bT  } p\,\bigg|
\le C \qquad\mbox{on}\quad\bT,
\label{19}
\end{equation}
where  $C$ depends only on $F$ and $a,b,c,(l_{ij}),(m_{ij})$.
\label{lem2}
\end{lemma}

\begin{proof}
We know that $m_{11}l_{22}=0$. If $l_{22}\ne0$, then $m_{11}=0$, and,
in view of \eqref{16newt}, $a+p_{xx}-tl_{11}p_x$ and
$b+p_{yy}-tl_{22}p_x-tm_{22}p_y$ are positive on $\bT$. Applying Lemma
\ref{lem1} with $(\alpha,\beta)$ equal to $(-tl_{11},-a)$ leads to
$|p_x|\le C$. Now we apply Lemma \ref{lem1} with 
\[(\alpha,\beta)=(-tm_{22},\ -b - \|l_{22}p_x\|_{ L^\infty })\] to obtain
$|p_y|\le C$.  If $m_{11}\neq 0$, then $l_{22}=0$, and obtain
(\ref{18}) similarly.  Estimate (\ref{19}) follows from (\ref{18}),
bearing in mind the definition \eqref{biperiod} of $\bT$.
\end{proof}

\begin{lemma} For $0\le t\le 1$, let $p\in C^2(\bT)$ be a solution of
{\rm\eqref{15newt}} and {\rm\eqref{16newt}}.  Then
$$
|\nabla ^2  p|
\le C \qquad \mbox{on}\quad\bT,
$$
where $C$ depends only on $F$ and $a,b,c,(l_{ij}),(m_{ij})$.

\label{lem3}
\end{lemma}

\begin{proof}
With Lemma \ref{lem2}, this follows from \cite[theorem 1]{S}.
\end{proof}

\begin{lemma} For $0\le t\le 1$, let $p\in C^2(\bT)$ be a solution of
{\rm\eqref{15newt}} and {\rm\eqref{16newt}}. Then
$p\in C^\infty(\bT)$ and for any positive integer $k$,
$$
\left\|\>p-\frac 1{|\bT|}\!\int_{  \bT  } p\,\right\|
_{C^k(\bT) } \le C,
$$
where $C$ depends only on $k$, $F$, and $a,b,c,(l_{ij}),(m_{ij})$.
\label{lem4}
\end{lemma}

\begin{proof}
With Lemmas \ref{lem2} and \ref{lem3}, it follows from a theorem of
Nirenberg \cite{N} that for some constant $0<\mu<1$, and some constant
$C$,
$$
\left\|\>p -  \frac 1{|\bT|} \int_{  \bT  } p\,\right\|_{  C^{2, \mu}(\bT) }
\le C,
$$
for all solutions of (\ref{15newt}) and (\ref{16newt}), and for all
$0\le t\le 1 $.  The higher derivative estimates then follow from
Schauder estimates.
\end{proof}

For $0<\mu<1$ and $0\le t\le 1$, let
$$
X^{2,\mu}:=\Big\{p\in C^{2, \mu}(\bT)\ |\
\int_{ \bT}p = 0\Big\},
$$
$$
X^\mu:=\Big\{q\in C^ \mu(\bT)\ |\
\int_{ \bT}q = 0\Big\},
$$
and
\begin{eqnarray*}
S_t(p)\kern-8pt &:=&\!\!
(a+p_{xx}-tl_{11}p_x-tm_{11}p_y)
(b+p_{yy}-tl_{22}p_x-tm_{22}p_y)\\[3pt]
&&\hskip50pt
-(c+p_{xy}-tl_{12}p_x-tm_{12}p_y)^2 -\big[t\eF+(1-t)(ab-c^2)\big].
\end{eqnarray*}
Integration by parts (\eqref{intparts} and the equations preceding
it) implies that
$$
S_t : X^{2,\mu}\to X^\mu.
$$
Clearly, for any $p\in X^{ 2, \mu}$,
$$
S_t'(p):  X^{2,\mu}\to X^\mu.
$$ Moreover, by elliptic theories, $S_t'(p)$ is an isomorphism from
$X^{2,\mu}$ to $X^\mu$ for any $p\in X^{ 2, \mu}$ satisfying
(\ref{16newt}).

At $t=0$, the zero function $p\equiv 0$ is a solution of
(\ref{15newt}). Since the linearized operator $S_t'(p)$ is an
isomorphism from $X^{2,\mu}$ to $X^\mu$, and since we have established
apriori estimates in Lemma \ref{lem4}, the solvability follows from
the standard method of continuity. In this way, Theorem \ref{thm1} is
established.

\section{The 2-step case revisited}\label{2fib}

We return to the study of case (1\1) and a discrete quotient
$M=\Gamma_0\backslash(\Nil^3\times S^1)$, with the invariant coframe
\eqref{basisKT} that we reproduce for convenience:
\[ e^1=dy,\qquad e^2=dx,\qquad e^3=dt,\qquad e^4=dz-x\kern1pt dy.\]
Consider the fibration
\[ \pi_{yt}\colon M\longrightarrow\bT^2_{yt},\]
where a volume form on the base is $e^{13}$.

As we remarked in Section~\ref{AK}, the fibres may or may not be
Lagrangian, depending on the choice of symplectic form. For example,
whilst $e^{14}+e^{23}$ restricts to zero, $e^{13}+e^{42}$ is
non-degenerate on the fibres of $\pi_{yt}$. We are unable to say much
about the situation of a symplectic fibration without the hypothesis
on $J$ that was used in Proposition~\ref{simp}. Next, we turn out
attention to the Lagrangian case.

Apply Lemma~\ref{filter} to retrieve a coframe $(f^i)$ respecting
the filtration determined by the ordered basis $(e^1,e^3,e^2,e^4)$,
and write
\[ df^4 = l f^{12}+mf^{13},\qquad m\ne0.\]
This corresponds to \eqref{df3df4}, except that $k=0$ in
the present context. One may
easily repeat the calculations in the proof of Theorem~\ref{3stepthm}
with $k$ (that played little role) zero. We obtain the same
generalized Monge--Amp\`ere equation
\[
\Big(\frac{B^2+C^2}{AB^2} + p_{yy} - \frac{mC^2}{A^2\!B}p_t\Big) 
\Big(\frac A{B^2} + p_{tt} - \frac mB p_t\Big) \! -
\Big(\!\!-\!\frac C{B^2} + p_{yt} + \frac{mC}{AB}p_t\Big)^2 
= \frac1{B^2}\eF\!,
\]
that can be solved by the techniques of Section~\ref{GMA}. In
this case the final step reduces to
\[ Al p_y + l Cp_t+A^2a_{2,y}-ABa_{1,t}=0,\]
which can be simply solved by setting
\[a_1 = \frac{lC}{AB}p,\qquad a_2 = -\frac lA p.\]
In conclusion,

\begin{cor}\label{secondpr}
Let $(g,J,\Om)$ be an invariant \aK structure on $M$ for which $\Om$
restricts to zero on the fibres of $\pi_{yt}$ over $\bT=\bT^2_{yt}$.
Let $\sigma=\eF\,\Om^2$ be a volume form with $F\in C^\infty(\bT)$ and
$\int_\bT(\eF-1)=0$. Then the associated Calabi--Yau problem has a
unique solution.
\end{cor}

\subsection{Manifolds belonging to (1\2) and (1\3)}\label{12}
Consider now a $T^2$ fibration
\[\pi\colon M=\Gamma\backslash X\to\mathbb{T}^2,\]
where $X=Nil^3\times \R$ and the generators of $\Gamma$ are given by
either (1\2) or (1\3). In this case, $\Gamma$ is not a lattice of $X$,
but it contains a lattice $\Gamma_0$ of $X$ such that
$\Gamma_0\backslash \Gamma=\Z_2$. Therefore there exists a
covering map $p\colon \Gamma_0\backslash X\to \Gamma\backslash X$
which preserves the $T^2$-bundle structure over $\mathbb{T}^2$. An \aK
structure in this context is called \emph{invariant} if it is
left-invariant on $X$ and invariant by $\Gamma$. In particular every
invariant \aK structure on $\Gamma\backslash X$ induces an
invariant \aK structure on $\Gamma_0\backslash X$. Since
$p$ preserves the $T^2$-bundle, Corollary \ref{secondpr} implies that
the Calabi--Yau problem for $T^2$-invariant volume form can be solved
even in these two cases.

\subsection{Representation theory.}\label{Thetafunc}
The fibrations $\pi_{xy},\pi_{zt}$ of the Kodaira--Thurston manifold
are intimately related to its function theory and unitary
representations of the Heisenberg group $\Nil^3$. We describe this
briefly.

For simplicity, ignore the $\R$ factor and work on the nilmanifold
$M^3=\Z^3\backslash\Nil^3$, where $\Z^3$ is the standard integer
lattice (see \eqref{xyz}). Consider the action of $\Nil^3$ on $f\in
L^2(\R)$ defined by setting
\[  (H_{x,y,z}\cdot f)(u)=\eE^{2\pi ik(z+yu)}f(x+u),\]
where $H_{x,y,z}$ is the matrix \eqref{xyz} and (here) $k$ is an
integer. This makes $L^2(\R)$ into a unitary representation of
$\Nil^3$ that we denote by $V_k$.

Following the discussion on \cite[page~6]{KU}, for any smooth function
$f\in L^2(\R)$, we set
\begin{equation}\label{Ff}
F(x,y,z) = \sum_{n\in\Z} (H_{x,y,z}\cdot f)(n)
= \eE^{2\pi ikz}\sum_{n\in\Z}f(x+n)\eE^{2\pi inky}.
\end{equation}
By its construction, $F=F_f$ is a well-defined function $M^3\to\C$.
Indeed \[F(x+a,y+b,z+c+ay) = F(x,y,z),\qquad a,b,c\in\Z,\] using the
coordinates of \eqref{xyz}. In particular, one may regard $x,y$ as
defined on $\bT^2_{xy}$ and $y$ as defined on $\bT^2_{yt}$, but we are
now considering functions outside the $T^2$ invariant class that we
have previously considered. If we take $f(x)=\eE^{-\pi x^2}$ then
\begin{equation}\label{Ftheta}
F(x,y,z) = e^{2\pi ikz}\vartheta(ky+ix)f(x)
\end{equation}
where $\vartheta$ is a classical theta function.

The mapping $f\mapsto F$ given by \eqref{Ff} realizes $V_k$ as a
summand of $L^2(M^3)$. A more subtle discussion leads to a Peter--Weyl
type decomposition
\begin{equation}\label{L2M}
L^2(M^3)\cong L^2(\bT^2)\oplus\bigoplus_{k\in\Z}|k|V_k,
\end{equation}
in which $V_k$ occurs as an isotypic component of multiplicity $|k|$
(see \cite{AB} and references therein).  In our context, one might
first hope to extend the class of assigned volume forms to functions
on $M^3$, and analyse their behaviour under the relevant differentuial
operators with the aid of \eqref{L2M}. Generalizations of the latter
hold for $L^2$ spaces of sections of holomorphic line bundles over the
Kodaira--Thurston manifold \cite{Eg,KU}.

\section{Almost-K\"ahler structures for $\Sol^3\times\R$}
\label{Sol}

In this final section, we make some observations regarding Calabi--Yau
problems for the manifolds belonging to the families (3\1) and
(3\2).

As in Section \ref{2fib}, we assume that $M$ is a solvmaniolfd
(i.e.\ that $\Gamma$ is a lattice) and then we use the observation of
Subsection \ref{12} to generalize to the case of
infra-solvmanifolds. Let
\[\pi\colon M=\Gamma \backslash \R^4\longrightarrow\bT_{zt}^2\] be the
relevant $T^2$-bundle with $\pi(x,y,z,t)=(z,t)$ and $\Gamma$ a lattice
of $G$. In the notation of Section~\ref{classification}, we begin with
the coframe
\begin{equation}\label{1341}\textstyle
e^1 = dt,\qquad e^2 = dz,\qquad e^3 = \eE^t\,dx,\qquad e^4 =
\eE^{-t}dy.
\end{equation}
Each tangent space to the fibres of $\pi_{zt}$ is represented by the
annihilator $\left<e^1,e^2\right>^\circ$, and $e^{12}$ is a volume
form on the base $\bT_{zt}^2$.

A special feature of the geometry modelled on \eqref{1341} is that the
transversal subspace $\left<e^3,e^4\right>^\circ$ is tangent to a
distribution $D$ that is also integrable. This is an immediate
consequence of the fact that the invariant forms $e^3,e^4$ generate a
differential ideal. Moreover,

\begin{lemma}\label{sol^3}
Let $(g,J,\Om)$ be any invariant \aK structure on $M$. The distribution
$D$ is $J$-holomorphic, and there exists an orthonormal basis $(f^i)$
such that
\begin{equation}\label{Omega4}
\Om=f^{12}+f^{34},
\end{equation}
and
\[ f^1\in \langle e^1\rangle,\qquad f^3\in\langle e^3\rangle,\qquad
f^4\in\langle e^3,e^4\rangle.\]
\end{lemma}

\begin{proof}
Let $(f^i)$ be an orthonormal basis of 1-forms for which
\eqref{Omega4} is valid and $f^1\in\langle e^1\rangle$.
Since any 2-form in the image of $d$ is divisible by $e^1$, we
get $d(f^{12})=-f^1\we df^2=0$. Thus,
\[0 = d(f^{34}) = -f^3\we df^4 +df^3\we f^4
= e^1\we (f^3\we\sigma+f^4\we\tau),\] for some 1-forms
$\sigma,\tau\in\langle e^3,e^4\rangle \cap \langle f^3,f^4\rangle$. If
$\sigma\we\tau=0$, then $\left<f^3,f^4\right>$ contains a closed
1-form that cannot be proportional to $e^1$. Then then
$d(f^{34})\ne0$. Therefore $\sigma,\tau$ are linearly independent and
$\langle f^3,f^4\rangle=\langle e^3,e^4\rangle$. Finally, we may
rotate in this plane so as to select $f^3\in \langle e^3\rangle$.

It now follows that $f^4=Jf^3$ and so $D$ is $J$-invariant.
\end{proof}

If $f^2\in\left<e^1,e^2\right>$ then $\pi_{zt}$ is also
$J$-holomorphic, and we can apply Proposition~\ref{simp} to obtain an
elementary solution:

\begin{cor}\label{dual1}
Let $(g,J,\Om)$ be an invariant \aK structure on $M$ for which
$\pi_{zt}$ is $J$-holomorphic, and let $\sigma=\eF\Om$ be a normalized
volume form with $\eF\in C^\infty(\bT^2_{zt})$. Then the Calabi--Yau
equation $(\Om+d\alpha)^2=\sigma$ has a unique solution.
\end{cor}

One can apply the same argument for a volume form that is
constant along the leaves of the associated foliation:

\begin{cor}\label{dual2}
Let $(g,J,\Om)$ be an invariant \aK structure on $M$ and let
$\sigma=\eF\Om$ be a normalized volume form such that $dF\we
e^{34}=0$. Then the Calabi--Yau equation $(\Om+d\alpha)^2=\sigma$ has
a unique solution.
\end{cor}

\begin{proof}
Let $(f^i)$ a coframe of $1$-forms as in Lemma \ref{sol^3}. A solution
of the Calabi--Yau equation is simply given by
$\tilde{\Om}=f^{12}+{\rm e}^F f^{34}$. Indeed, $\tilde \Om$ is
compatible with respect to $J$ since $J$ preserves $f^{34}$, and
$\tilde\Om-\Om=({\rm e}^F-1)f^{34}$ is exact since the
normalization \[ \int_M(\eF-1)f^{1234}=0\] implies that it has zero
cup product with $H^2(M,\R)=\langle [f^{12}],[f^{34}]\rangle =\langle
[e^{12}],[e^{34}]\rangle$.
\end{proof}

It is interesting to compare the dual situations highlighted by the
last two results. In Corollary~\ref{dual1}, the choice of $\Omega$
(or, alternatively, $J$) is restricted but $F$ is a free function on
the base. On the other hand, Corollary~\ref{dual2} refers to an
arbitrary invariant \aK structure, though $\Omega$ and $J$ are in
practice already constrained by the geometry. Because of the anture of
the foliation one is dealing with, the applicable class of functions
is restricted, as our final example illustrates.

\begin{ex}
This is a sequel to Example~\ref{Ex1} in which $M=M^3\times S^1$, with
$z$ a coordinate on $S^1$. The leaves of $D$ have the form $C\times
S^1$, where $C$ is an integral curve of the form
\[ (x,y,t)=(x_0/\sigma^k,\>y_0\sigma^k,\>t),\qquad k\delta\le t< (k+1)\delta.\]
For generic $(x_0,y_0)$ this curve is dense in $M^3$ and so any smooth
function constant on $C$ is constant on $M^3$. Therefore the
hypothesis $dF\we e^{34}=0$ implies that $F\in C^\infty(S^1)$, and the
normalization is $\int_{S^1}F(z)\,dz=0$. 
\end{ex}

\subsection{Conclusion} Earlier in the paper, we successfully solved
the Calabi--Yau problem for volume forms invariant by a 2-torus
action, relative to various fixed almost-K\"ahler structures and
Lagrangian fibrations to a 2-torus base $\bT^2$. The $T^2$-invariance
is a natural hypothesis, in view of the analogy with toric geometry in
which moment mappings play the role of the fibrations.

On the other hand, the examples in the present section show that, in a
more general context, it is futile to restrict the class of volume
forms. The manifold $\Gamma\backslash\Sol^3$ exhibits geometry of a
very different type, involving an action by the $2\times2$ diagonal
matrix $\varphi(t)$ (see Section~\ref{classification}) also seen on
the unit tangent bundle
$\mathit{PSL}(2,\Z)\backslash\mathit{PSL}(2,\R)$ of the modular
surface \cite{Ghys}.

The theory described in Subsection~\ref{Thetafunc}, or rather its
generalizations to four dimensions tailored to symplectic geometry
\cite{Eg,KU}, is likely to be relevant in solving the problem in a
general non-Lagrangian setting.

In another direction, one can vary the ambient \aK structure. As a
first step, the result of Corollary~\ref{cor45} can be extended to
`separable' symplectic forms such as
$\Om(x,y)=f(x)e^{14}+g(y)e^{23}$. Such calculations lead one to
postulate further Calabi--Yau esistence results.

\bigbreak

\bigbreak\vfil

\footnotesize\parindent0pt\parskip8pt

Anna Fino\\
Dipartimento di Matematica, Universit\`a di Torino, Via
Carlo Alberto 10, 10123 Torino, Italia\\
\texttt{annamaria.fino@unito.it}

YanYan Li\\
Department of Mathematics, Rutgers University,
110 Frelinghuysen Road, Piscataway, NJ 08854, USA\\
\texttt{yyli@math.rutgers.edu}

Simon Salamon\\
Dipartimento di Matematica, Politecnico di Torino,
Corso Duca degli Abruzzi 24, 10129 Torino, Italia\\
\texttt{simon.salamon@polito.it}

Luigi Vezzoni\\
Dipartimento di Matematica, Universit\`a di Torino,
Via Carlo Alberto 10, 10123 Torino, Italia\\
\texttt{luigi.vezzoni@unito.it}

\enddocument